\documentclass[reqno, 12pt]{amsart}

\usepackage[a4paper=true,pdfpagelabels]{hyperref}
\usepackage{graphicx}

\usepackage{caption}
\usepackage{subcaption}

\usepackage[ansinew]{inputenc}
\usepackage{amsfonts,epsfig}
\usepackage{latexsym}
\usepackage{amsmath}
\usepackage{amssymb}
\usepackage{comment}

\newtheorem{theorem}{Theorem}

\newtheorem{lemma}[theorem]{Lemma}

\theoremstyle{definition}

\newtheorem{example}[theorem]{Example}

\theoremstyle{remark}

\numberwithin{equation}{section}

\newcommand{\set}[1]{\left\{#1\right\}}
\newcommand{\abs}[1]{\left|#1\right|}
\newcommand{\brac}[1]{\left(#1\right)}
\newcommand{\sqbrac}[1]{\left[#1\right]}
\newcommand{\lbrac}[1]{\left\langle#1\right\rangle}
\newcommand{\nm}[1]{\lVert#1\rVert}

\newcommand{\D}{\mathbb{D}}

\newcommand{\slicht}{\mathcal{S}}
\newcommand{\N}{\mathbb{N}}
\newcommand{\R}{\mathbb{R}}

\newcommand{\B}{\mathcal{B}}
\newcommand{\NO}{\mathcal{N}}
\newcommand{\C}{\mathbb{C}}
\newcommand{\riem}{\overline{\mathbb{C}}}

\newcommand{\e}{\varepsilon}

\renewcommand{\phi}{\varphi}

\newcommand{\uloc}{U_{\textrm{loc}}^M}
\newcommand{\uloca}{U_{\textrm{loc}}^A}
\newcommand{\uloch}{U_{\textrm{loc}}^H}
\DeclareMathOperator{\Real}{Re}

\def\VMOA{\mathord{\rm VMOA}}

\def\d{\delta}           \def\e{\varepsilon}
           
       \def\t{\theta}       
         \def\r{\rho}         \def\z{\zeta}

\begin{document}

\title[On Becker's univalence criterion]{On Becker's univalence criterion}

\author{Juha-Matti Huusko}
\address{Department of Physics and Mathematics, University of Eastern Finland, P.O. Box 111, FI-80101 Joensuu, Finland}
\email{juha-matti.huusko@uef.fi}
\thanks{This research was supported in part by the Academy of Finland \#268009, and the Faculty of Science and Forestry of the University of Eastern Finland \#930349.}

\author{Toni Vesikko}
\address{Department of Physics and Mathematics, University of Eastern Finland, P.O. Box 111, FI-80101 Joensuu, Finland}
\email{tonive@student.uef.fi}



\date{\today}


\begin{abstract}
We study locally univalent functions $f$ analytic in the unit disc~$\mathbb{D}$ of the complex plane such that $\left|{f''(z)/f'(z)}\right|(1-|z|^2)\leq 1+C(1-|z|)$ holds for all $z\in\mathbb{D}$, for some $0<C<\infty$. If $C\leq 1$, then $f$ is univalent by Becker's univalence criterion. We discover that for $1<C<\infty$ the function $f$ remains to be univalent in certain horodiscs. Sufficient conditions which imply that $f$ is bounded, belongs to the Bloch space or belongs to the class of normal functions, are discussed. Moreover, we consider generalizations for locally univalent harmonic functions.
\end{abstract}

\maketitle



\section{Introduction}
\label{sc:introduction}

Let $f$ be meromorphic in the unit disc $\D=\set{z\in\mathbb{C} \,:\, |z|<1}$ of the complex plane~$\mathbb{C}$. Then $f$ is locally univalent, denoted by $f\in\uloc$, if and only if its spherical derivative $f^\#(z)=|f'(z)|/(1+|f(z)|^2)$ is non-vanishing. Equivalently, the Schwarzian derivative
    $$
    S(f)=\brac{\frac{f''}{f'}}'-\frac12\brac{\frac{f''}{f'}}^2
    $$
of $f$ is an analytic function.  If $z_0\in\D$ is a pole of $f$, we define $f^\#(z_0)=\lim_{w\to z}f^\#(w)$ and $S(f)(z_0)=\lim_{w\to z_0}S(f)(w)$ along $w\in\D$ where $f(w)\neq 0$. Both the Schwarzian derivative $S(f)$ and the pre-Schwarzian derivative $P(f)=f''/f'$ can be derived from the Jacobian $J_f=|f'|^2$ of $f$, namely
\begin{equation}
\label{eq:jacobian}
P(f)=\frac{\partial}{\partial z}(\log J_f),\quad S(f)=P(f)'-\frac12 P(f)^2.
\end{equation}

According to the famous Nehari univalence criterion~\cite[Theorem~1]{N:1949}, if $f\in\uloc$ satisfies
\begin{equation}
\label{eq:nehari}
\abs{S(f)(z)}(1-|z|^2)^2\leq N,\quad z\in\D,
\end{equation}
for $N=2$, then $f$ is univalent. The result is sharp by an example by Hille~\cite[Theorem~1]{H:1922}.


Binyamin Schwarz~\cite{S:1955} showed that if $f(a)=f(b)$ for some $a\neq b$ for $f\in\uloc$, then
\begin{equation}
\label{eq:schwarz}
\max_{\z\in\lbrac{a,b}}\abs{S(f)(\z)}(1-|\z|^2)^2>2.
\end{equation}
Here $\lbrac{a,b}=\set{\varphi_a(\varphi_a(b)t)\,:\, 0\leq t\leq 1}$
is the hyperbolic segment between $a$ and $b$ and
\begin{equation}
\label{eq:autom}
\varphi_a(z)=\frac{a-z}{1-\overline{a}z}
\end{equation}
is an automorphism of the unit disc. Condition~\eqref{eq:schwarz} implies that if
\begin{equation}
\label{eq:schwarz2}
\abs{S(f)(z)}(1-|z|^2)^2\leq N,\quad r_0\leq |z|<1,
\end{equation}
for $N=2$ and some $0<r_0<1$, then $f$ has finite valence~\cite[Corollary~1]{S:1955}. If \eqref{eq:schwarz2} holds for $N<2$, then $f$ has a spherically continuous extension to~$\overline{\D}$, see~\cite[Theorem~4]{GP:1984}.

Chuaqui and Stowe~\cite[p.~564]{CS:2008} asked whether
\begin{equation}
\label{eq:chuaqui}
\abs{S(f)(z)}(1-|z|^2)^2\leq 2+C(1-|z|),\quad z\in\D,
\end{equation}
where $0<C<\infty$ is a constant, implies that $f$ is of finite valence. The question remains open despite of some progress achieved in~\cite{GR:2015}. Steinmetz~\cite[p.~328]{Steinmetz:1983} showed that if~\eqref{eq:chuaqui} holds, then $f$ is normal, that is, the family $\set{f\circ\varphi_a\,:\, a\in\D}$ is normal in the sense of Montel. Equivalently, $\nm{f}_{\NO}=\sup_{z\in\D}f^\#(z)(1-|z|^2)<\infty$.

A function $f$ analytic in~$\D$ is locally univalent, denoted by $f\in\uloca$, if and only if $J_f=|f'|^2$ is non-vanishing. By the Cauchy integral formula, if $g$ is analytic in~$\D$, then
    $$
    |g'(z)|(1-|z|^2)^2\leq 4\max_{|\z|=\frac{1+|z|^2}{2}}|g(\z)|(1-|\z|^2),\quad z\in\D.
    $$
Consequently, the inequality
    $$
    \nm{S(f)}_{H^\infty_2}\leq 4\nm{P(f)}_{H^\infty_1}+\frac12\nm{P(f)}_{H^\infty_1}^2
    $$
holds. Here, we denote $\nm{g}_{H_p^\infty}=\sup_{z\in\D}|g(z)|(1-|z|^2)^p$ for $0<p<\infty$. Thus, each one of the conditions~\eqref{eq:nehari},~\eqref{eq:schwarz2} and~\eqref{eq:chuaqui} holds if $|f''(z)/f'(z)|(1-|z|^2)$ is sufficiently small for $z\in\D$. Note also that conversely
    $$
    \nm{P(f)}_{H^\infty_1}\leq 2+2\sqrt{1+\frac12\nm{S(f)}_{H^\infty_2}},
    $$
see~\cite[p.~133]{Pommerenke:1964}.

The famous Becker univalence criterion~\cite[Korollar~4.1]{Becker:1972} states that if $f\in\uloca$ satisfies
\begin{equation}
\label{eq:becker}
    \abs{zP(f)}(1-|z|^2)\leq \rho,\quad z\in\D,
\end{equation}
for $\rho\leq 1$, then $f$ is univalent in~$\D$, and if $\rho<1$, then $f$ has a quasi-conformal extension to~$\riem= \C \cup\set{\infty}$. For $\rho>1$, condition~\eqref{eq:becker} does not guarantee the univalence of $f$~\cite[Satz~6]{BP:1984} which can in fact break brutally~\cite{Gevirtz}. If~\eqref{eq:becker} holds for $0<\rho<2$, then $f$ is bounded, and in the case $\rho=2$, $f$ is a Bloch function, that is, $\nm{f}_\B=\sup_{z\in\D}|f'(z)|(1-|z|^2)<\infty$.

Becker and Pommerenke proved recently that if
\begin{equation}
\label{eq:BP}
\abs{\frac{f''(z)}{f'(z)}}(1-|z|^2)< \rho,\quad r_0\leq |z|<1,
\end{equation}
for $\rho<1$ and some $r_0\in(0,1)$, then $f$ has finite valence~\cite[Theorem~3.4]{BP:2016}. However, the case of equality $\rho=1$ in~\eqref{eq:BP} is open and the sharp inequality corresponding to~\eqref{eq:schwarz}, in terms of the pre-Schwarzian, has not been found yet.

In this paper, we consider the growth condition
\begin{equation}
\label{eq:HV}
    \abs{\frac{f''(z)}{f'(z)}}(1-|z|^2)\leq 1+C(1-|z|),\quad z\in\D,
\end{equation}
where $0<C<\infty$ is an absolute constant, for $f\in\uloca$. When~\eqref{eq:HV} holds, we detect that $f$ is univalent in horodiscs $D(ae^{i\theta},1-a)$, $e^{i\theta}\in\partial\D$, for some $a=a(C)\in [0,1)$. Here $D(a,r)=\set{z\in \C \,:\, |z-a|<r}$ is the Euclidean disc with center $a\in \C $ and radius $0<r<\infty$.

The remainder of this paper is organized as follows. In Section~\ref{sc:Distortion}, we see that under condition~\eqref{eq:HV} the function $f\in\uloca$ is bounded. Weaker sufficient conditions which imply that the function $f$ is either bounded, a Bloch function or a normal function are investigated. The main results concerning univalence are stated in Section~\ref{sc:main} and proved in Section~\ref{sc:proofs}.  Finally in Section~\ref{sc:Generalizations for harmonic} we state generalizations of our results to harmonic functions. Moreover, for sake of completeness, we discuss the harmonic counterparts of the results proven in~\cite{GR:2015}.

\section{Distortion theorems}
\label{sc:Distortion}

Recall that each meromorphic and univalent function $f$ in~$\D$ satisfies~\eqref{eq:nehari} for $N=6$. This is the converse of Nehari's theorem, discovered by Kraus~\cite{Kraus:1932}. In the same fashion, each analytic and univalent function $f$ in~$\D$ satisfies
\begin{equation}
\label{eq:koebe}
\abs{\frac{zf''(z)}{f'(z)}-\frac{2|z|^2}{1-|z|^2}}\leq\frac{4|z|}{1-|z|^2},\quad z\in\D,
\end{equation}
and hence~\eqref{eq:becker} holds for~$\rho=6$, which is the converse of Becker's theorem~\cite[p.~21]{P:1975}.

The class $\slicht$ consists of functions $f$ univalent and analytic in~$\D$ such that $f(0)=0$ and $f'(0)=1$. Among all functions in $\slicht$, the Koebe function
    $$
    k(z)=\frac{z}{(1-z)^2}=\frac{1}{(1-z)^2}-\frac{1}{1-z},
    $$
has the extremal growth. Namely, by inequality~\eqref{eq:koebe}, each $f\in\slicht$ satisfies
\begin{equation}
\label{eq:kobe5}
    |f^{(j)}(z)|\leq k^{(j)}(|z|),
    \quad
    \abs{\frac{f^{(j+1)}(z)}{f^{(j)}(z)}}
    \leq \frac{k^{(j+1)}(|z|)}{k^{(j)}(|z|)},\quad j=0,1,
\end{equation}
for $z\in\D\setminus\set{0}$ and $j=0,1$. Moreover, $k$ satisfies condition~\eqref{eq:nehari}, for~$N=6$, with equality for each~$z\in\D$.

Bloch and normal functions emerge in a natural way as Lipschitz mappings. Denote the Euclidean metric by $d_E$, and define the hyperbolic metric in~$\D$ as
    $$
    d_{H}(z,w)=\frac12\log\frac{1+|\varphi_z(w)|}{1-|\varphi_z(w)|},\quad z,w\in\D,
    $$
where $\varphi_z(w)$ is defined as in~\eqref{eq:autom}, and the chordal metric in~$\riem= \C \cup \set{\infty}$ by setting
    $$
    \chi(z,w)=\frac{|z-w|}{\sqrt{1+|z|^2}\sqrt{1+|w|^2}},\quad \chi(z,\infty)=\frac{1}{\sqrt{1+|z|^2}},\quad z\in\riem,\, w\in \C .
    $$
Then each $f\in\B$ is a Lipschitz function from $(\D,d_H)$ to $( \C ,d_E)$ with a Lipschitz constant equal to $\nm{f}_{\B}$, and each $f\in\NO$ is a Lipschitz map from $(\D,d_H)$ to $(\riem,\chi)$ with constant $\nm{f}_\NO$. To see the first claim, assume that $f$ is analytic in $\D$ such that
    $$
    |f(z)-f(w)|\leq M d_H(z,w),\quad z,w\in\D.
    $$
By letting $w\to z$, we obtain $|f'(z)|(1-|z|^2)\leq M$, for all $z\in\D$, and conclude that $\nm{f}_\B\leq M$. Conversely, if $f\in\B$, then
\begin{equation*}
\label{B-normi lips}
\begin{split}
    |f(z)-f(w)|
    \leq\int_{\lbrac{z,w}}|f'(\zeta)||d\zeta|
    \leq\sup_{\zeta\in\lbrac{z,w}}|f'(\zeta)|(1-|\zeta|^2)d_H(z,w),
\end{split}
\end{equation*}
and we conclude that $f$ is a Lipschitz map with a constant $M\leq\nm{f}_\B$.

In the same fashion as above, we deduce that
\begin{equation*}
\label{eq:becker-BC}
    \abs{\frac{f''(z)}{f'(z)}}\leq\frac{B}{1-|z|^2}+\frac{C(1-|z|)}{1-|z|^2},\quad z\in\D,
\end{equation*}
for some $0<B,C<\infty$, is equivalent to
    $$
    \abs{\log\frac{f'(z)}{f'(w)}}
    \leq B d_H(z,w)+C\brac{1-\frac{|z+w|}{2}+\frac{|z-w|}{2}}d_H(z,w),\quad z,w\in\D.
    $$
This follows from the fact that the hyperbolic segment $\lbrac{z,w}$ is contained in the disc $D\brac{(z+w)/2,|z-w|/2}$, which yields
    $$
    1-|\z|\leq 1-\frac{|z+w|}{2}+\frac{|z-w|}{2},\quad \z\in\lbrac{z,w}.
    $$

We may deduce some relationships between the classes $\B$ and $\NO$. By the Schwarz-Pick lemma, each bounded analytic function belongs to $\B$. If $f\in\B$, then both $f\in\NO$ and $e^f\in\NO$. This is clear, since $\chi(z,w)\leq d_E(z,w)$ for all $z,w\in \C $ and since the exponential function is Lipschitz from  $(\C ,d_E)$ to $(\riem , \chi)$. Moreover, since each rational function $R$ is Lipschitz from $(\riem,\chi)$ to itself, $R\circ f\in\NO$ whenever $f\in\NO$. However, it is not clear when $f^2\in\NO$ implies $f\in\NO$.

If $f\in\uloc$ is univalent, then both $f,f'\in\NO$ by the estimate
    $$
    (f^{(j)})^\# (z)
    =\frac{|f^{(j+1)}(z)|}{1+|f^{(j)}(z)|^2}\leq\frac{1}{2}\abs{\frac{f^{(j+1)}(z)}{f^{(j)}(z)}}
    $$
and~\eqref{eq:kobe5}. However, it is not clear if $f''\in\NO$. At least, each primitive $g$ of an univalent function satisfies $g''\in\NO$. Recently, similar normality considerations which have connections to differential equations, were done in~\cite{Grohn:2017}.

If $f\in\uloca$ and there exists $0<\delta<1$ such that $f$ is univalent in each pseudo-hyperbolic disc $\Delta(a,\delta)=\set{z\in\D \, : \, |\varphi_a(z)|<\delta}$, for $a\in\D$, then $f$ is called uniformly locally univalent. By a result of Schwarz, this happens if and only if $\sup_{z\in\D}|S(f)(z)|(1-|z|^2)^2<\infty$, or equivalently if $\log f'\in\B$. Consequently, the derivative of each uniformly locally univalent function is normal.

By using arguments similar to those in the proof of~\cite[Theorem~3.2]{BP:2016} and in~\cite{KS:2002}, we obtain the following result.

\begin{theorem}
\label{th:distortion3}
Let $f$ be meromorphic in~$\D$ such that
\begin{equation}
\label{eq:cond3}
\abs{\frac{f''(z)}{f'(z)}}\leq\varphi(|z|),\quad 0\leq R\leq |z|<1,
\end{equation}
for some $\varphi:[R,1)\to[0,\infty)$.
\begin{enumerate}
\item [\rm(i)] If
\begin{equation}
\label{eq:phi-cond1}
\limsup_{r\to 1^-}(1-r)\exp\brac{\int_R^r\varphi(t)\,dt}<\infty,
\end{equation}
then $\displaystyle \sup_{R<|z|<1}|f'(z)|(1-|z|^2)<\infty$.
\item [\rm(ii)] If
\begin{equation}
\label{eq:phi-cond2}
\int_{R}^1\exp\brac{\int_R^s\varphi(t)\,dt}\,ds<\infty,
\end{equation}
then $\displaystyle \sup_{R<|z|<1}|f(z)|<\infty$.
\end{enumerate}
\end{theorem}
\begin{proof}
Let $\zeta\in\partial\D$. Let $R\leq\rho<r<1$ and note that $f'$ is non-vanishing on the circle $|z|=\rho$. Then
    $$
    \abs{\log\frac{f'(r\zeta)}{f'(\rho\zeta)}}
    \leq\int_\rho^r\abs{\frac{f''(t\zeta)}{f'(t\zeta)}}\,dt
    \leq\int_\rho^r\varphi(t)\,dt.
    $$
Therefore
    $$
    |f'(r\zeta)|\leq|f'(\rho\zeta)|\exp\brac{\int_\rho^r\varphi(t)\,dt},
    $$
which implies the first claim. By another integration,
    $$
    |f(r\zeta)-f(\rho\zeta)|\leq|f'(\rho\zeta)|\int_\rho^r\exp\brac{\int_\rho^s\varphi(t)\,dt}\,ds.
    $$
Hence,
    $$
    |f(z)|\leq M(\rho,f)
    +M(\rho,f')\int_{\rho}^1\exp\brac{\int_\rho^s\varphi(t)\,dt}\,ds<\infty
    $$
for $\rho<|z|<1$.
\end{proof}

The assumptions in Theorem~\ref{th:1}(i) and~(ii) are satisfied, respectively, by the functions
    $$
    \varphi(t)=\frac{2}{1-t^2}=\brac{\log\frac{1+t}{1-t}}
    $$
and
    $$
    \psi(t)=\frac{B}{1-t^2}+\frac{C}{1-t^2}\brac{\log\frac{e}{1-t}}^{-(1+\e)},
    $$
where $0<\e<\infty$, $0<B<2$ and $0<C<\infty$.


By Theorem~\ref{th:distortion3}, if $f$ is meromorphic in~$\D$ and satisfies~\eqref{eq:cond3} and~\eqref{eq:phi-cond1} for some~$\varphi$, then $f\in\NO$. Moreover, if $f$ is also analytic in~$\D$, then $f\in\B$, and if~\eqref{eq:phi-cond2} holds, then $f$ is bounded.

\section{Main results}
\label{sc:main}

Next we turn to present our main results. We consider Becker's condition in a neighborhood of a boundary point $\zeta\in\partial\D$ as well as univalence in certain horodiscs. Furthermore, we state some distortion type estimates similar to the converse of Becker's theorem. Some examples which concerning the main results and the distribution of preimages of a locally univalent function are discussed.

\begin{theorem}\label{Thm-Becker-Local}
Let $f\in\uloca$ and $\z\in\partial\D$.

If there exists a sequence $\{w_n\}$ of points in $\D$ tending to $\z$ such that
    \begin{equation}\label{Eq:LimitInLocalization}
    \left|\frac{f''(w_n)}{f'(w_n)}\right|(1-|w_n|^2)\to c
    \end{equation}
for some $c\in(6,\infty]$, then
for each $\d>0$ there
exists a point $w\in f(\D)$ such that at least two of its distinct
preimages belong to $D(\z,\d)\cap\D$.

Conversely, if for each $\d>0$ there
exists a point $w\in f(\D)$ such that at least two of its distinct
preimages belong to $D(\z,\d)\cap\D$, then
there exists a sequence $\{w_n\}$ of points in $\D$ tending to $\z$ such that \eqref{Eq:LimitInLocalization} holds for some $c\in[1,\infty]$.
\end{theorem}

\begin{example}
It is clear that \eqref{Eq:LimitInLocalization}, $c\in(6,\infty)$, does not imply that $f$ is of infinite valence. For example, the polynomial $f(z)=(1-z)^{2n+1}$, $n\in\N$, satisfies the sharp inequality
    $$
    \left|\frac{f''(z)}{f'(z)}\right|(1-|z|^2)\le 4n,\quad z\in\D,
    $$
although $f(z)=\e^{2n+1}$ has $n$ solutions in $D(1,\d)\cap\D$ for each $\e\in(0,\d)$ when $\delta\in(0,1)$ is small enough (depending on $n$).

The function $f(z)=(1-z)^{-p}$, $0<p<\infty$, satisfies the sharp inequality
    $$
    \left|\frac{f''(z)}{f'(z)}\right|(1-|z|^2)\le2(p+1),\quad z\in\D,
    $$
and for each $p\in(2n,2n+2]$, $n\in\N\cup\{0\}$, the valence of $f$ is $n+1$ for suitably chosen points in the image set.
\end{example}

Under the condition~\eqref{eq:HV}, function the $f$ is bounded, see Theorem~\ref{th:distortion3} in Section~\ref{sc:Distortion}. Condition~\eqref{eq:HV} implies that $f$ is univalent in horodiscs.

\begin{theorem}\label{th:1}
Let $f\in\uloca$ and assume that~\eqref{eq:HV} holds
for some $0<C<\infty$. If $0<C\leq 1$, then $f$ is univalent in~$\D$. If $1<C<\infty$, then there exists $0<a<1$, $a=a(C)$, such that $f$ is univalent in all discs $D(ae^{i\theta},1-a)$, $0\leq \theta<2\pi$. In particular, we can choose $a=1-(1+C)^{-2}$.
\end{theorem}

Let $f\in\uloca$ be univalent in each horodisc $D(ae^{i\theta},1-a)$, $0\leq \theta<2\pi$, for some $0<a<1$. By the proof of~\cite[Theorem~6]{GR:2015}, for each $w\in f(\D)$, the sequence of pre-images $\set{z_n}\in f^{-1}(w)$ satisfies
\begin{equation}
\label{eq:carleson}
    \sum_{z_n\in Q}(1-|z_n|)^{1/2}\leq K\ell(Q)^{1/2}
\end{equation}
for any Carleson square $Q$ and some constant $0<K<\infty$ depending on $a$. Here
    $$
    Q=Q(I)
    =\set{re^{i\theta}\,:\, e^{i\theta}\in I,\; 1-\frac{|I|}{2\pi}\leq r<1}
    $$
is called a Carleson square based on the arc $I\subset\partial\D$ and $|I|=\ell(Q)$ is the Euclidean arc length of $I$.

By choosing $Q=\D$ in~\eqref{eq:carleson}, we obtain
\begin{equation*}
\label{eq:counting}
n(f,r,w)=O\left(\frac{1}{\sqrt{1-r}}\right),\quad r\to1^-,
\end{equation*}
where $n(f,r,w)$ is the number of pre-images $\set{z_n}= f^{-1}(w)$ in the disc $\overline{D(0,r)}$. Namely, arrange $\set{z_n}= f^{-1}(w)$ by increasing modulus, and let $0<|z_n|=r<|z_{n+1}|$ to deduce
    $$
    (1-r)^{1/2}n(f,r,w)\leq\sum_{k=0}^n(1-|z_k|)^{1/2}\leq K\ell(\D)^{1/2}<\infty
    $$
for some $0<K(a)<\infty$.

\begin{theorem}\label{th:3}
Let $f\in\uloca$ be univalent in all Euclidean discs
    $$
    D\brac{\frac{C}{1+C}e^{i\theta},\frac{1}{1+C}},\quad e^{i\theta}\in\partial\D,
    $$
for some $0<C<\infty$. Then
    $$
    \abs{\frac{f''(z)}{f'(z)}}(1-|z|^2)\leq 2+4(1+K(z)),\quad z\in\D,
    $$
where $K(z)\asymp(1-|z|^2)$ as $|z|\to 1^-$.
\end{theorem}

In view of~\eqref{eq:koebe}, Theorem~\ref{th:3} is sharp. Moreover, since~\eqref{eq:koebe} implies
    $$
    \abs{\frac{f''(z)}{f'(z)}}(1-|z|)\leq\frac{4+2|z|}{1+|z|}\leq 4
    $$
for univalent analytic functions $f$, the next theorem is sharp as well.

\begin{theorem}\label{th:2}
Let $f\in\uloca$ be univalent in all Euclidean discs
    $$
    D(ae^{i\theta},1-a)\subset\D,\quad e^{i\theta}\in\partial\D,
    $$
for some $0<a<1$. Then
\begin{equation}\label{eq:thb1}
    \abs{\frac{f''(z)}{f'(z)}}(1-|z|)\leq 4,\quad a\leq|z|<1.
\end{equation}
\end{theorem}

\begin{example}
Let $f=f_{C,\zeta}$ be a locally univalent analytic function in $\D$ such that $f(-1)=0$ and
\begin{equation*}
\label{examplef}
    f'(z)=-i\left(\frac{1+z}{1-z}\right)^\frac12e^{\frac{C\z z}2},\quad \z\in\partial\D,\, z\in\D.
\end{equation*}
Then
    $$
    \frac{f''(z)}{f'(z)}=\frac1{1-z^2}+\frac{C\z}{2},
    $$
hence~\eqref{eq:HV} holds
and $f$ is univalent in $\D$ if $C\le1$ by Becker's univalence criterion. If $f$ is univalent, then $f-f(0)\in\mathcal{S}$ and we obtain for $\zeta=1$,
		$$
		1\geq \frac{f'(x)}{k'(x)}=\frac{e^{\frac{Cx}{2}}(1-x)^{5/2}}{(1+x)^{1/2}}
		\sim\frac{1+Cx/2}{1+3x},\quad x\to 0^+.
		$$
Therefore, if $C>6$, then $f$ is not univalent.

The boundary curve $\partial f(\D)$ has a cusp at $f(-1)=0$. When $\zeta=-i$, the cusp has its worst behavior, and by numerical calculations the function $f$ is not univalent if $C>2.21$. Moreover, as $C$ increases, the valence of $f$ increases, see Figure~\ref{fig:cc}.

The curve $\set{f(e^{it})\,:\, t\in(0,\pi]}$ is a spiral unwinding from $f(-1)$. We may calculate the valence of $f$ by counting how many times $h(t)=\Real(f(e^{it}))$ changes its sign on $(0,\pi]$. Numerical calculations suggest that the valence of $f$ is approximately equal to $\frac{100}{63}C$.

\begin{figure}[h!]
\centering
\begin{subfigure}[b]{0.475\textwidth}
\centering
\includegraphics[width=\textwidth]{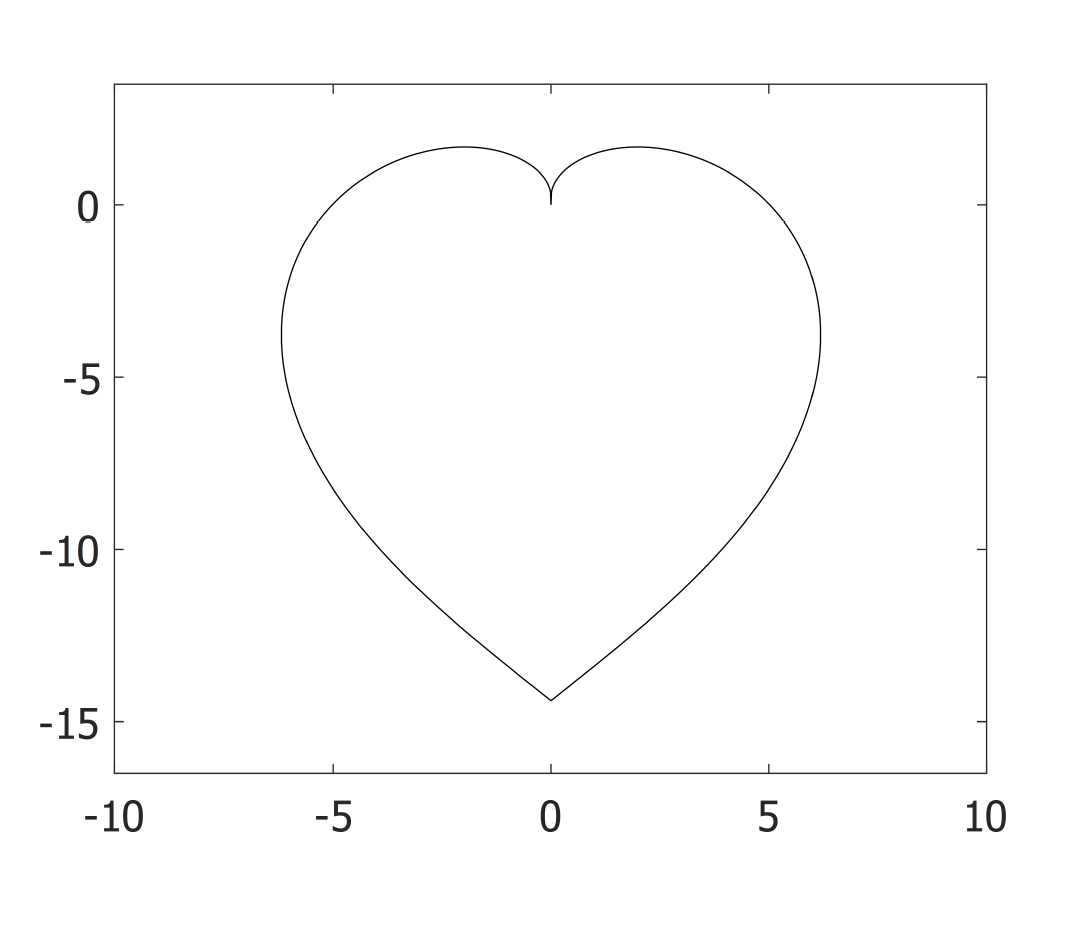}
\caption{$f(\D)$ for $C=2.21$ and $\zeta=-i$.}
\label{fig:c2}
\end{subfigure}
\begin{subfigure}[b]{0.475\textwidth}
\centering
\includegraphics[width=\textwidth]{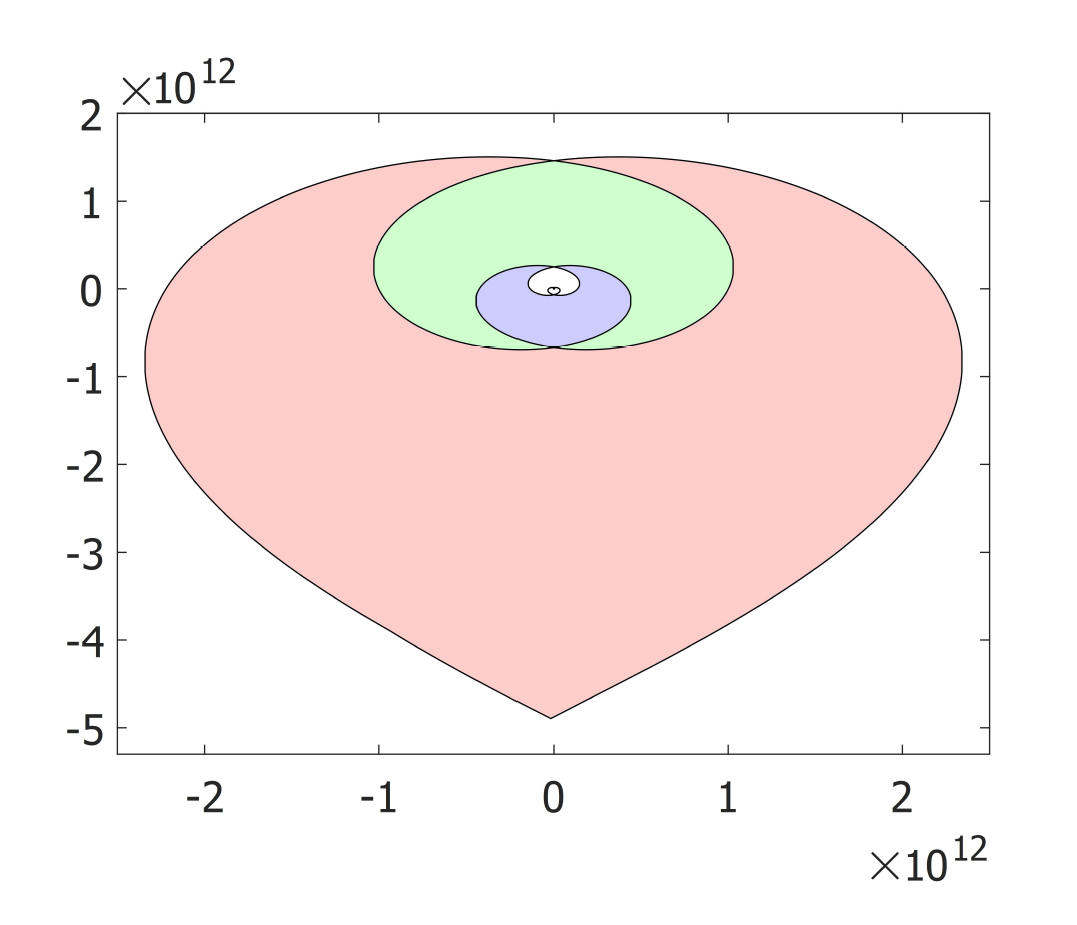}
\caption{$f(\D)$ for $C=30$ and $\zeta=-i$.}
\label{fig:c30}
\end{subfigure}
\caption{Image domain $f(\D)$ for different values of $C$. In~(A), $\partial f(\D)$ is a simple closed curve. In~(B), the valences of red, green and blue parts of $f(\D)$, under $f$, are one, two and three, respectively.}
\label{fig:cc}
\end{figure}
\end{example}

\section{Proofs of main results}
\label{sc:proofs}

In this section, we proof the results stated in Section~\ref{sc:main}.

\begin{proof}[Proof of Theorem~\ref{Thm-Becker-Local}]
To prove the first assertion, assume on the contrary that there
exists $\d>0$ such that $f$ is univalent in $D(\z,\d)\cap\D$. Without loss of generality, we may assume that $\z=1$. Let $T$ be a conformal map of $\D$ onto a domain $\Omega\subset
D(\z,\d)\cap\D$ with the following properties:
    \begin{enumerate}
    \item[\rm(i)] $T(\z)=\z$; \item[\rm(ii)]$\partial\Omega\supset
    \{e^{i\t}:|\arg\z-\t|<t\}$ for some $t>0$; \item[\rm(iii)]
    $\left|\frac{T''(z)}{T'(z)}\right|(1-|z|^2)^\frac12\le 1-\r$ for
    all $z\in\D$, where $\rho\in(0,1)$ is any pregiven number.
    \end{enumerate}
The existence of such a map follows, for instance, by
\cite[Lemma~8]{GGPPR:2013}. Then
    $$
    \left|\frac{f''(T(z))}{f'(T(z))}T'(z)+\frac{T''(z)}{T'(z)}\right|(1-|z|^2)\le6,\quad z\in\D,
    $$
by~\eqref{eq:koebe}, since $f\circ T$ is univalent in~$\D$. Moreover,
$\frac{T''(z)}{T'(z)}(1-|z|^2)\to0$, as $|z|\to1^-$, by (iii). Let
$\{w_n\}$ be a sequence such that $w_n\to\z$, and define $z_n$ by
$T(z_n)=w_n$. Then $z_n\to\z$, and since $T'$ belongs to the disc
algebra by \cite[Lemma~8]{GGPPR:2013}, we have
    $$
    1<\frac{1-|T(z_n)|^2}{|T'(z_n)|(1-|z_n|^2)}\to1^+,\quad n\to\infty.
    $$
For more details, see~\cite[p.~879]{GR:2015}.
It follows that
    \begin{equation*}
    \begin{split}
    &\limsup_{n\to\infty}\,\left|\frac{f''(w_n)}{f'(w_n)}\right|(1-|w_n|^2)\\
    &=\limsup_{n\to\infty}\,\left|\frac{f''(T(z_n))}{f'(T(z_n))}\right|(1-|T(z_n)|^2)\\
    &=\limsup_{n\to\infty}\,\left|\frac{f''(T(z_n))}{f'(T(z_n))}\right||T'(z_n)|(1-|z_n|^2)
    \frac{(1-|T(z_n)|^2)}{|T'(z_n)|(1-|z_n|^2)}\le6,
    \end{split}
    \end{equation*}
which is the desired contradiction.

To prove the second assertion, assume on the contrary that
\eqref{Eq:LimitInLocalization} fails, so that there exist
$\r\in(0,1)$ and $\d\in(0,1)$ such that
    \begin{equation}\label{1}
    \left|\frac{f''(z)}{f'(z)}\right|(1-|z|^2)\le\r,\quad z\in D(\z,\d)\cap\D.
    \end{equation}
 If $g=f\circ T$, then \eqref{1} and
(i)--(iii) yield
    \begin{eqnarray*}
    \left|\frac{g''(z)}{g'(z)}\right|(1-|z|^2)
    &\le&\left|\frac{f''(T(z))}{f'(T(z))}\right||T'(z)|(1-|z|^2)+\left|\frac{T''(z)}{T'(z)}\right|(1-|z|^2)\\
    &\le&\left|\frac{f''(T(z))}{f'(T(z))}\right|(1-|T(z)|^2)+1-\r\le1
    \end{eqnarray*}
for all $z\in\D$. Hence $g$ is univalent in $\D$ by Becker's
univalence criterion, and so is $f$ on $\Omega$. This is clearly a
contradiction.
\end{proof}

\begin{proof}[Proof of Theorem~\ref{th:1}]
Assume that condition~\eqref{eq:HV} holds for some $0<C\leq 1$. Now
    $$
    \abs{\frac{zf''(z)}{f'(z)}}(1-|z|^2)\leq |z|(1+C(1-|z|))\leq |z|+1-|z|=1,
    $$
and hence $f$ is univalent in~$\D$ by Becker's univalence criterion.

Assume that~\eqref{eq:HV} holds for some $1<C<\infty$. It is enough to consider the case $\theta=0$. Let $T(z)=a+(1-a)z$ for $z\in\D$, and $g=f\circ T$. Then
\begin{equation*}\label{eq:tha1}
\begin{split}
(1-|z|^2)\abs{\frac{g''(z)}{g'(z)}}
&= (1-|z|^2)\abs{\frac{f''(T(z))}{f'(T(z))}}|T'(z)|\\
&=\abs{\frac{f''(T(z))}{f'(T(z))}}(1-|T(z)|^2)\frac{(1-|z|^2)|T'(z)|}{1-|T(z)|^2}\\
&\leq\brac{1+C(1-|T(z)|)}\frac{(1-|z|^2)(1-a)}{1-|T(z)|^2}\\
&\leq(1+C(1-|a+(1-a)z|))\frac{(1-|z|^2)(1-a)}{1-|a+(1-a)z|^2}.
\end{split}
\end{equation*}
By the next lemma, for $a=1-(1+C)^{-2}$, $g$ is univalent in~$\D$ and $f$ is univalent in $D(a,1-a)$. The assertion follows.
\end{proof}

\begin{lemma}\label{lm:th1}
Let $1<C<\infty$. Then, for $z\in\D$,
    $$
    \brac{1+C\brac{1-\abs{\frac{C^2+2C}{C^2+2C+1}+\frac{1}{(1+C)^2}z}}}
    \times
    \frac{(1-|z|^2)\frac{1}{(1+C)^2}}{1-\abs{\frac{C^2+2C}{C^2+2C+1}+\frac{1}{(1+C)^2}z}^2}
    \leq 1.
    $$
\end{lemma}
\begin{proof}
Let $h:[0,1)\to\R$, be defined by $h(t)=(1+C(1-t))/(1-t^2)$. Then
    $$
    h'(t)=\frac{-Ct^2+2(1+C)t-C}{(1-t^2)^2}=0
    $$
if and only if $t=t_C=\frac{1+C-\sqrt{1+2C}}{C}\in(0,1)$. Hence, $h$ is strictly decreasing on $[0,t_C]$ and strictly increasing on $[t_C,1]$. If
    $$
    t=\abs{\frac{C^2+2C}{C^2+2C+1}+\frac{1}{(1+C)^2}z}\leq t_C,
    $$
then
\begin{equation*}\label{eq:lma1}
\begin{split}
h(t)(1-|z|^2)\frac{1}{(1+C)^2}
&\leq h(0)(1-|z|^2)\frac{1}{(1+C)^2}
\leq \frac{1}{1+C}\leq 1.
\end{split}
\end{equation*}
On the other hand, if
    $$
    t_C<t=\abs{\frac{C^2+2C}{C^2+2C+1}+\frac{re^{i\theta}}{(1+C)^2}}\leq\frac{C^2+2C+r}{C^2+2C+1}=t',
    $$
then we obtain
\begin{equation}\label{eq:lma2}
\begin{split}
h(t)\frac{(1-|z|^2)}{(1+C)^2}
\leq h(t')\frac{1-r^2}{(1+C)^2}
=\frac{(1+C)^2+C(1-r)}{2(1+C)^2-(1-r)}(1+r)\leq 1,
\end{split}
\end{equation}
provided that
    $$
    k_C(r)=(1+r)\sqbrac{(1+C)^2+C(1-r)}+1-r\leq 2(1+C)^2.
    $$
Since $k_C(1)\leq 2(1+C)^2$ and
    $$
    k_C'(r)=(1+C)^2+C(1-r)-C(1+r)-1>0
    $$
for $r<1+C/2$, inequality~\eqref{eq:lma2} holds. This ends the proof of the lemma.
\end{proof}

\begin{proof}[Proof of Theorem~\ref{th:3}]
Let $a\in\D$, $0<C/(1+C)<|a|<1$ and $g(z)=f(\varphi_a(r_az))$,
where $\varphi_a(z)$ is defined as in~\eqref{eq:autom}. Moreover, let
    $$
    r_a^2=\frac{|a|-\frac{C}{1+C}}{|a|\brac{1-|a|\frac{C}{1+C}}}.
    $$
The pseudo-hyperbolic disc $\Delta_p(\alpha,\rho)=\set{z\in\D\,:\, |\varphi_\alpha(z)|\leq\rho}$ with center $\alpha\in\D$ and radius $0<\rho<1$ satisfies
    $$
    \Delta_p(\alpha,\rho)
		=D\brac{\xi(\alpha,\rho),R(\alpha,\rho)}
    =D\brac{\frac{1-\rho^2}{1-|\alpha|^2\rho^2}\alpha,\frac{1-|\alpha|^2}{1-|\alpha|^2\rho^2}\rho}.
    $$
We deduce
    $$
    \Delta_p(a,r_a)
		=
		D\brac{\frac{a}{|a|}\frac{C}{1+C},R(a,r_a)}
		\subset
		D\brac{\frac{a}{|a|}\frac{C}{1+C},\frac{1}{1+C}},
    $$
so that $g$ is univalent in~$\D$. Now
    $$
    \frac{g''(0)}{g'(0)}
    =\frac{f''(a)}{f'(a)}\varphi_a'(0)r_a+\frac{\varphi_a''(0)}{\varphi_a'(0)}r_a
    =-\frac{f''(a)}{f'(a)}(1-|a|^2)r_a+2\overline{a}r_a.
    $$
By~\eqref{eq:koebe}, $|g''(0)/g'(0)|\leq 4$ and therefore
    $$
    \abs{\frac{f''(a)}{f'(a)}(1-|a|^2)-2\overline{a}}\leq\frac{4}{r_a},
    $$
which implies
    $$
    \abs{\frac{f''(a)}{f'(a)}}(1-|a|^2)\leq 2+\frac{4}{r_a}=2+4(1+K(a)),
    $$
where
    $$
K(a)=\frac{1}{r_a}-1
=\frac{1-r_a^2}{r_a(1+r_a)}
\sim \frac{1}{2}(1-r_a^2)
=\frac{1}{2}\frac{\frac{C}{1+C}(1-|a|^2)}{|a|\brac{1-|a|\frac{C}{1+C}}}\sim \frac{C}{2}(1-|a|^2),
    $$
as $|a|\to 1^-$.
\end{proof}

\begin{proof}[Proof of Theorem~\ref{th:2}]
It suffices to prove~\eqref{eq:thb1} for $|z|=a$, since trivially $f$ is univalent also in $D(be^{i\theta},1-b)\subset D(ae^{i\theta},1-a)$ for $a<b<1$ and $e^{i\theta}\in\partial\D$. Moreover, by applying a rotation $z\mapsto\lambda z$, $\lambda\in\partial\D$, it is enough to prove~\eqref{eq:thb1} for $z=a$.

Let $T(z)=a+(1-a)z$ for $z\in\D$. Now $g=f\circ T$ is univalent in~$\D$ and by~\eqref{eq:koebe}
    $$
    \abs{\frac{f''(a)}{f'(a)}}(1-a)
    =\abs{\frac{f''(T(0))}{f'(T(0))}}|T'(0)|
    =\abs{\frac{g''(0)}{g'(0)}}\leq 4.
    $$
The assertion follows.
\end{proof}

\section{Generalizations for harmonic functions}
\label{sc:Generalizations for harmonic}

Let $f$ be a complex-valued and harmonic function in~$\D$. Then $f$ has the unique representation $f=h+\overline{g}$, where both $h$ and $g$ are analytic in~$\D$ and $g(0)=0$. In this case, $f=h+\overline{g}$ is orientation preserving and locally univalent, denoted by $f\in\uloch$, if and only if its Jacobian $J_f=|h'|^2-|g'|^2>0$, by a result by Lewy~\cite{Lewy}. In this case, $h\in\uloca$ and the dilatation $\omega_f=\omega=g'/h'$ is analytic in~$\D$ and maps~$\D$ into itself. Clearly $f=h+\overline g$ is analytic if and only if the function $g$ is constant.

For $f=h+\overline{g}\in\uloch$, equation~\eqref{eq:jacobian} yields the harmonic pre-Schwarzian and Schwarzian derivatives:
    $$
    P(f)
    =P(h)-\frac{\overline{\omega}\, \omega'}{1-|\omega|^2}.
    $$
and
    $$
    S(f)=S(h)+\frac{\overline \omega}{1-|\omega|^2} \left(\frac{h''}{h'}\, \omega'-\omega''\right)-\frac 32 \left(\frac{\overline \omega\, \omega'}{1-|\omega|^2}\right)^2.
    $$
This generalization of $P(f)$ and $S(f)$ to harmonic functions was introduced and motivated in~\cite{HM-Schwarzian}.

There exists $0<\delta_0<2$ such that if $f\in\uloch$ satisfies~\eqref{eq:nehari} for $N=\delta_0$, then $f$ is univalent in~$\D$, see~\cite{AW:1962} and~\cite{HM-qc}. The sharp value of $\delta_0$ is not known. Moreover, if $f\in\uloch$ satisfies
		$$
		\abs{P(f)}(1-|z|^2)+\frac{|\omega'(z)|(1-|z|^2)}{1-|\omega(z)|^2}\leq 1,\quad z\in\D,
		$$
then $f$ is univalent. The constant $1$ is sharp, by the sharpness of Becker's univalence criterion. If one of these mentioned inequalities, with a slightly smaller right-hand-side constant, holds in an annulus $r_0<|z|<1$, then $f$ is of finite valence~\cite{HM:2017}.

Conversely to these univalence criteria, there exist absolute constants $0<C_1,C_2<\infty$ such that if $f\in\uloch$ is univalent, then~\eqref{eq:nehari} holds for $N=C_1$ and~\eqref{eq:becker} holds for $\rho=C_2$, see~\cite{HM-2015}. The sharp values of~$C_1$ and~$C_2$ are not known.

By the above-mentioned analogues of Nehari's criterion, Becker's criterion and their converses, we obtain generalizations of the results in this paper for harmonic functions. Of course, the correct operators and constants have to be involved. Theorem~\ref{Thm-Becker-Local} and its analogue~\cite[Theorem~1]{GR:2015} for the Schwarzian derivative $S(f)$ are valid as well. Moreover, Theorems~\ref{th:1},~\ref{th:3}, and~\ref{th:2} are valid. We leave the details for the interested reader.

We state the important generalization of~\cite[Theorem~3]{GR:2015} for harmonic functions here. It gives a sufficient condition for the Schwarzian derivative of $f\in\uloch$ such that the preimages of each $w\in f(\D)$ are separated in the hyperbolic metric. Here $\xi(z_1,z_2)$ is the hyperbolic midpoint of the hyperbolic segment $\lbrac{z_1,z_2}$ for $z_1,z_2\in\D$.

\begin{theorem}
\label{th:harmonic-preimage}
Let $f=h+\overline{g}\in\uloch$ such that
\begin{equation*}
\label{eq:harmonic-chuaqui}
    |S_H(f)|(1-|z|^2)\leq\delta_0(1+C(1-|z|)),\quad z\in\D,
\end{equation*}
for some $0<C<\infty$. Then each pair of points $z_1,z_2\in\D$ such that $f(z_1)=f(z_2)$ and $1-|\xi(z_1,z_2)|<1/C$ satisfies
\begin{equation}
\label{harm-separation}
    d_H(z_1,z_2)\geq
    \log\frac{2-C^{1/2}(1-|\xi(z_1,z_2)|)^{1/2}}
    {C^{1/2}(1-|\xi(z_1,z_2)|)^{1/2}}.
\end{equation}
Conversely, if there exists a constant $0<C<\infty$ such that each pair of points $z_1,z_2\in\D$ for which $f(z_1)=f(z_2)$ and $1-|\xi(z_1,z_2)|<1/C$ satisfies~\eqref{harm-separation}, then
    $$
    |S_H(f)|(1-|z|^2)\leq C_2(1+\Psi_C(|z|)(1-|z|)^{1/3}),\quad 1-|z|<(8C)^{-1},
    $$
where $\Psi_C$ is positive, and satisfies $\Psi_C(|z|)\to (2(8C)^{1/3})^+$ as $|z|\to 1^-$.
\end{theorem}

We have not found a natural criterion which would imply that $f=h+\overline{g}\in\uloch$ is bounded. However, the inequality $|g'(z)|<|h'(z)|$ can be utilized. A domain $D\subset\C$ is starlike if for some point $a\in D$ all linear segments $[a,z]$, $z\in D$, are contained in~$D$. Let $h\in\uloca$ be univalent, let $h(\D)$ be starlike with respect to $z_0\in h(\D)$ and $f=h+\overline{g}\in\uloch$. Then the function
    $$
    z\mapsto\Omega(z)=\frac{g(z)-g(z_0)}{h(z)-h(z_0)}
    $$
maps~$\D$ into~$\D$. To see this, let $a\in\D$ and let~$R=h^{-1}([h(z_0),h(a)])$ be the pre-image of the segment $[h(z_0),h(a)]$ under $h$. Then
    $$
    |h(a)-h(z_0)|=\int_R|h'(\zeta)||d\zeta|\geq\abs{\int_R g'(\zeta)\,d\zeta}=|g(a)-g(z_0)|.
    $$
Consequently, if $f=h+\overline{g}\in\uloch$ is such that $h(\D)$ is starlike and bounded, then $f(\D)$ is bounded.





\vspace{0.5cm}


\end{document}